\documentclass[12pt]{amsart}
\usepackage{enumerate,amsmath,amssymb}
\usepackage{amsthm}

\newtheorem{theorem}{Theorem}[section]
\newtheorem{lemma}[theorem]{Lemma}

\newtheorem{proposition}[theorem]{Proposition}

\theoremstyle{definition}
\newtheorem{definition}[theorem]{Definition}

\allowdisplaybreaks[2]

\begin{document}
\title[Krein parameters]{Krein parameters of fiber-commutative coherent configurations}
\author{Keiji Ito}
\address{Research Center for Pure and Applied Mathematics\\
Graduate School of Information Sciences\\
Tohoku University}
\email{k.ito@ims.is.tohoku.ac.jp}

\author{Akihiro Munemasa}
\address{Research Center for Pure and Applied Mathematics\\
Graduate School of Information Sciences\\
Tohoku University}
\email{munemasa@math.is.tohoku.ac.jp}

\date{December 13, 2019}

\maketitle
\begin{abstract}
For fiber-commu\-tative coherent configurations, we show that
Krein parameters can be defined essentially uniquely. 
As a consequence, the general
Krein condition reduces to positive semidefiniteness of finitely
many matrices determined by the parameters of a coherent
configuration. We mention its implications in the coherent
configuration defined by a generalized quadrangle. We also
simplify the absolute bound using the matrices of Krein parameters.
\end{abstract}

\section{Introduction}

Coherent configurations are defined by D. G. Higman in \cite{Hi1}.
A special class of coherent configurations, called homogeneous coherent configurations,
are also known as association schemes. The Krein condition asserts that
Krein parameters of a commutative association scheme
are non-negative real numbers (see \cite[Theorem~3.8]{BI}), 
and it can rule out the existence of some
putative association schemes. A generalization of this property was formulated
by Hobart \cite{Ho}, who proved that certain matrices determined by the
parameters are positive semidefinite. However, some complication arises due to 
the fact that an analogue of Krein parameters cannot be defined uniquely. This
seems to be an obstacle
to develop a nice theory for coherent configurations parallel to
commutative association schemes. 
In this paper, we restrict ourselves to fiber-commu\-tative coherent configurations.
This restriction enables us to define a basis of matrix units almost uniquely
for each simple two-sided ideals of the adjacency algebra $\mathfrak{A}$ of the
coherent configuration. Since the algebra $\mathfrak{A}$ is closed with respect
to the entry-wise product, we can define Krein parameters as the expansion coefficients
of the entry-wise product of two basis elements. These Krein parameters can then be
collected in a matrix form, which we call the matrix of Krein parameters, and
we show that this is a positive semidefinite hermitian matrix. This is a
generalization of the Krein condition for commutative association schemes to
fiber-commu\-tative coherent configurations. Our main theorem asserts that, the
general Krein condition formulated by Hobart \cite{Ho} which in general consists
of infinitely many inequalities, is equivalent to the positive semidefiniteness
of the finitely many matrices of Krein parameters. As an illustration, we consider
the fiber-commu\-tative coherent configurations defined by generalized quadrangles.
We write down all matrices of Krein parameters, and show that one cannot derive
any consequence other than the well-known inequalities established in \cite{Hi3,Hi2}.
Moreover, we also simplify the absolute bounds for fiber-commu\-tative
coherent configurations due to \cite{HW}.

This paper is organized as follows. 
In Section~2, we prepare notation and formulate the Krein condition for
coherent configurations.
In Section~3,
we define matrices of Krein parameters for
fiber-commu\-tative coherent configurations, and give our main theorem.
In Section~4,
we compute the matrices of Krein parameters for the coherent configuration
obtained from a generalized quadrangle.
In Section~5,
we apply our result to simplify the absolute bounds for
fiber-commu\-tative coherent configurations.

\section{Preliminaries}

For a finite set $X$, we denote by $\mathrm{M}_X(\mathbb{C})$
the algebra of square matrices with entries in $\mathbb{C}$ whose rows and
columns are indexed by $X$. We also denote by $J_X$ the all-ones matrix
in $\mathrm{M}_X(\mathbb{C})$.

\begin{definition}\label{dfn:coco}
Let $X= \coprod_{i =1}^f X_i$ be a partition of a finite set $X$.
For all pairs $i,j \in\{1,\dots, f\}$, let 
$X_i \times X_j=\coprod_{k=1}^{r_{i,j}} R_{i,j,k}$ be a partition, and let
$\mathcal{I} =\{(i,j, k) \mid 1 \leq i,j \leq f, 1 \leq k \leq r_{i,j} \}$.
For $I\in\mathcal{I}$, let $A_{I}\in\mathrm{M}_X(\mathbb{C})$ 
denote the adjacency matrix of the relation $R_{i,j,k}\subset X\times X$.
The pair $\mathcal{C} = (\mathcal{X}, \{R_I\}_{I \in \mathcal{I}} )$ is
called a {\it coherent configuration} if the following conditions hold:
\begin{enumerate}
\item
For each $i \in \{1 ,\dots , f\}$, there exists $k \in \{1 , \dots , r_{i,i} \}$
such that $A_{i,i,k}= I_{X_i}$, where $I_{X_i}$ is the $\{0,1\}$-matrix indexed
by $X \times X$ with $1$ on $(x,x)$-entry for $x \in X_i$ and $0$ otherwise.
\item
$\sum_{I \in \mathcal{I}} A_I = J_X$.
\item
For any $I \in \mathcal{I}$, there exists $I' \in \mathcal{I}$
such that $A_I^T=A_{I'}$.
\item
For any $I,J \in \mathcal{I}$,
$A_I A_J = \sum_{K \in \mathcal{I}} p_{I,J}^K A_K$ for some scalars $p_{I,J}^K$.
\end{enumerate}  
Each subset $X_i \subset X$ is called
a {\it fiber}, $|\mathcal{I}|$ is called the {\it rank}, and
$p_{I,J}^K$ are called the {\it parameters} of $\mathcal{C}$.
\end{definition}

For the remainder of this section, we fix a coherent configuration $\mathcal{C}$
as in Definition~\ref{dfn:coco}.
Let $\mathfrak{A}$ be the subalgebra of $\mathrm{M}_X(\mathbb{C})$
spanned by $\{A_I \mid I \in \mathcal{I}\}$.
This algebra is called the {\it adjacency algebra} of $\mathcal{C}$.
The subspace $\mathfrak{A}_{k,l}$ is defined as the subspace consisting of
matrices whose entries are zero except those indexed by $X_k \times X_l$.
For brevity, we write $\mathfrak{A}_k= \mathfrak{A}_{k,k}$.
Each $\mathfrak{A}_k$ forms the adjacency algebra of a coherent configuration with
single fiber, that is, an association scheme on $X_k$.

\begin{definition}
The coherent configuration $\mathcal{C}$ is said to be {\it fiber-commu\-tative}
if the algebra $\mathfrak{A}_k$ is commutative for all $k \in \{1, \dots ,f \}$.
Similarly, $\mathcal{C}$ is said to be {\it fiber-symmetric}
if the algebra $\mathfrak{A}_k$ consists only of symmetric matrices
 for all $k \in \{1, \dots ,f \}$.
\end{definition}

Let $\{ \Delta_s \mid s \in S \}$ be a set of representatives of
all irreducible matrix representations of $\mathfrak{A}$ over $\mathbb{C}$
 satisfying $\Delta_s(A)^* = \Delta_s(A^*)$ for any $A \in \mathfrak{A}$,
where $*$ denotes the transpose-conjugate.
Since $\mathfrak{A}$ is semisimple,
$\mathfrak{A}$ is completely reducible.
In other words, $\mathfrak{A}$ is decomposed into 
\[
\mathfrak{A} = \bigoplus_{s \in S} \mathfrak{C}_s,
\]
where $\mathfrak{C}_s$ is a simple two-sided ideal affording $\Delta_s$. 
Moreover, for each $s \in S$, $\Delta_s |_{\mathfrak{C}_s}$ is an isomorphism
from $\mathfrak{C}_s$ to $\mathrm{M}_{e_s}(\mathbb{C}) $.
This implies that
there exists a basis $\{ \varepsilon_{i,j}^s \in \mathfrak{A} \mid i,j  \in F_s\}$
of $\mathfrak{C}_s$ satisfying
\begin{align} \label{eq_mat_u}
\varepsilon_{i,j}^s \varepsilon_{k,l}^s 
&= \delta_{j,k} \varepsilon_{i,l}^s,\\ 
{\varepsilon_{i,j}^s}^* &= \varepsilon_{j,i}^s, \label{eq_mat_u2} 
\end{align}
where $|F_s|=e_s$.
Note that there is a good reason not to take $F_s=\{1, \dots , e_s \}$.
This will become clear after Lemma~\ref{dim=1}.
By \cite[Theorem~8]{HW},
we can choose $\varepsilon_{i,j}^s$ in such a way that
\begin{equation} \label{HWThm8}
\varepsilon_{i,j}^s \in \bigcup_{k,l=1}^f \mathfrak{A}_{k,l}
\quad(i,j\in F_s,\;s\in S).
\end{equation}
Note that, since $\mathfrak{A}_{k,l} \mathfrak{A}_{k',l'}=0$
if $l \neq k'$, \eqref{HWThm8} implies
\begin{equation} \label{diag_block}
\varepsilon_{i,i}^s \in \bigcup_{k=1}^f \mathfrak{A}_k
\quad(i\in F_s,\;s\in S).
\end{equation}
This is also mentioned in the proof of \cite[Theorem~8]{HW}.

Since $\overline{\mathfrak{C}_s}$ is also a simple two-sided ideal,
there exists $\hat{s} \in S$ such that
$\mathfrak{C}_{\hat{s}} = \overline{\mathfrak{C}_s}$.
If $\mathcal{C}$ is fiber-symmetric, then $s = \hat{s}$ for all $s \in S$
by \eqref{diag_block}.
Note that
$\{\overline{\varepsilon_{i,j}^s} \mid i,j \in F_s \}$ is a basis
of $\mathfrak{C}_{\hat{s}}$ satisfying \eqref{eq_mat_u}.
Since $\overline{\mathfrak{A}_{k,l}} = \mathfrak{A}_{k,l}$
for all $k,l \in \{1, \dots , f \}$,
\[
\overline{ \varepsilon_{i,j}^s} \in
\bigcup_{k,l=1}^f \mathfrak{A}_{k,l}
\quad(i,j\in F_s,\;s\in S).
\]
This implies that we can choose $\{\varepsilon_{i,j}^s\mid i,j\in F_s\}$
and $\{\varepsilon_{i,j}^{\hat{s}}\mid i,j\in F_s\}$
in a manner compatible with complex conjugation.

\begin{definition}\label{def_mat_un}
For each $s \in S$, a basis $\{\varepsilon_{i,j}^s \mid i,j \in F_s \}$
of $\mathfrak{C}_s$ is called {\it a basis of matrix units
for $\mathfrak{C}_s$} if \eqref{eq_mat_u} and \eqref{HWThm8} hold.
If $\{\varepsilon_{i,j}^s \mid i,j \in F_s \}$ is a basis of matrix units
for $\mathfrak{C}_s$ for each $s \in S$,
then their union is called {\it bases of matrix units for $\mathfrak{A}$}
provided that $F_s=F_{\hat{s}}$ and
\[
\overline{\varepsilon_{i,j}^s} = \varepsilon_{i,j}^{\hat{s}} \quad (i,j \in F_s , s \in S).
\]
\end{definition}

Note that bases of matrix units are not determined uniquely (see \cite{Hi1}),
but we will see later that they are essentially unique
for the fiber-commu\-tative case.

\begin{lemma}\label{Z(A)}
The center of $\mathfrak{A}$ is contained in
$\bigoplus_{k=1}^f \mathfrak{A}_k$.
\end{lemma}
\begin{proof}
This is immediate from \eqref{diag_block},
since $\sum_{i \in F_s} \varepsilon_{i,i}^s$ is the central idempotent
corresponding to $\mathfrak{C}_s$.
\end{proof}

Let $J_{k,l}$ be the matrix in $\mathfrak{A}$ with $1$ in all entries
indexed by $X_k \times X_l$ and $0$ otherwise.
Without loss of generality,
we may assume that
$\mathfrak{C}_1 = \mathfrak{A} \varepsilon_1 \mathfrak{A}$,
where
\begin{equation}
\varepsilon_1
=  \sum_{k=1} \frac{1}{|X_k|} J_{k,k}.
\label{centralidem}
\end{equation}
This implies that we may also assume that
\begin{equation}
\varepsilon_{k,l}^{1} = \frac{1}{\sqrt{|X_k||X_l|}} J_{k,l}
\label{matunit}
\end{equation}
for any $k,l \in  F_1$, where $F_1=\{1, \dots, f\}$.

For the reminder of this section, we fix bases of matrix units
$\{\varepsilon_{i,j}^s \mid s \in S,\; i,j \in F_s \}$ for $\mathfrak{A}$.
Let 
$
\Lambda_s = F_s^2 \times \{s\}
$
for each $s \in S$ and
$
\Lambda = \coprod_{s \in S} \Lambda_s.
$
Moreover, we denote
$\varepsilon_\lambda = \varepsilon_{i,j}^s $
for $\lambda =(i,j,s) \in \Lambda$.
Define $n_{\lambda}= \sqrt{|X_k||X_l|}$,
where $\lambda \in \Lambda$ and $\varepsilon_{\lambda} \in \mathfrak{A}_{k,l}$.
Let $\circ$ denote the Hadamard (entry-wise) product of matrices.
Since $\mathfrak{A}$ is closed with respect to $\circ$,
there exist $q_{\lambda, \mu}^\nu \in \mathbb{C}$ such that
\begin{align}
n_{\lambda} \varepsilon_\lambda \circ n_{\mu} \varepsilon_\mu
= \sum_{\nu \in \Lambda} q_{\lambda, \mu}^\nu n_{\nu} \varepsilon_\nu. \label{def_krein}
\end{align}

\begin{definition}
The complex numbers $q_{\lambda, \mu}^\nu $ appearing in \eqref{def_krein}
are called {\it Krein parameters with respect to bases of matrix units 
$ \{\varepsilon_\lambda \mid \lambda \in \Lambda \}$.} 
\end{definition}

Let $\mathcal{P}_F$ denote the set of
the all positive semidefinite hermitian matrices in $\mathrm{M}_F( \mathbb{C})$.

\begin{theorem}[Krein conditions {\cite[Lemma~1]{Ho}}]\label{krein}
For any $s,t,u \in S$, 
$B = (b_{i,j}) \in \mathrm{M}_{F_s}(\mathbb{C})$ and
$C =(c_{i,j}) \in \mathrm{M}_{F_t}(\mathbb{C})$,
let $\tilde{Q}_{s,t}^u(B,C)$ denote the matrix in $\mathrm{M}_{F_u}(\mathbb{C})$
whose $(m,n)$-entry is
\begin{equation}
\sum_{i,j \in F_s} \sum_{k,l \in F_t} b_{i,j} c_{k,l}
q_{(i,j,s),(k,l,t)}^{(m,n,u)} . \label{tilde}
\end{equation}
Then 
\begin{equation} \label{krein_con}
\tilde{Q}_{s,t}^u(B,C) \in \mathcal{P}_{F_u} \quad  
(B \in \mathcal{P}_{F_s},\; C \in \mathcal{P}_{F_t}).
\end{equation}
\end{theorem}

Let $\eta$ be the mapping
from $\Lambda$ to $\{1, \dots, f \}^2$
defined by $\varepsilon_\lambda \in \mathfrak{A}_{\eta(\lambda)}$
for $\lambda \in \Lambda$,
or equivalently,
\begin{equation}
\eta(i,j,s) = (k,l) \text{ if } \varepsilon_{i,j}^s \in \mathfrak{A}_{k,l}.
\label{def_eta}
\end{equation}

\begin{lemma}\label{lem:F'}
For each $s\in S$, define
\[F'_s=\{k \mid 1 \leq k \leq f,  (k,k) \in \{\eta(i,i,s)\mid i \in F_s\} \}.\]
Then $\eta(\Lambda_s) = {F'_s}^2$.
\end{lemma}
\begin{proof}
First, we prove $\eta(\Lambda_s) \subset {F'_s}^2$.
For $(i,j,s) \in \Lambda_s$, suppose $\eta(i,j,s)=(k,l)$.
Namely, $\varepsilon_{i,j}^s \in \mathfrak{A}_{k,l}$.
By \eqref{eq_mat_u} and \eqref{diag_block},
$\varepsilon_{i,i}^s \in \mathfrak{A}_k$
and $\varepsilon_{j,j}^s \in \mathfrak{A}_l$ hold.
Thus $\eta(i,i,s) =(k,k)$ and $\eta(j,j,s)=(l,l)$ and
these mean $k,l \in F'_s$.

Conversely, suppose $\eta(i,i,s)=(k,k)$ and $\eta(j,j,s)=(l,l)$,
where $i,j \in F_s$.
Then $\varepsilon_{i,i}^s \in \mathfrak{A}_k$ and
$\varepsilon_{j,j}^s \in \mathfrak{A}_l$.
By \eqref{eq_mat_u}, we obtain $\varepsilon_{i,j}^s \in \mathfrak{A}_{k,l}$.
Thus $(k,l)=\eta(i,j,s) \in \eta(\Lambda_s)$.
\end{proof}

\begin{lemma}\label{zero}
Let $\lambda, \mu, \nu \in \Lambda$.
If $q_{\lambda, \mu}^\nu \neq 0$,
then $\eta (\lambda) = \eta(\mu) =\eta(\nu)$.
\end{lemma}
\begin{proof}
By the definition of $\eta$,
$\varepsilon_\lambda \in \mathfrak{A}_{\eta(\lambda)}$, 
$\varepsilon_\mu \in \mathfrak{A}_{\eta(\mu)}$, and
$\varepsilon_\nu \in \mathfrak{A}_{\eta(\nu)}$
hold.
If $\eta(\lambda) \neq \eta(\mu)$,
then $\varepsilon_\lambda \circ \varepsilon_\mu =0$,
and this means $q_{\lambda, \mu}^\nu =0$ for any $\nu \in \Lambda$.
If $\eta(\lambda)= \eta(\mu) \neq \eta(\nu)$,
then
$\varepsilon_\lambda \circ \varepsilon_\mu \in \mathfrak{A}_{\eta(\lambda)}$ and
this means that $q_{\lambda,\mu}^\nu=0$.
\end{proof}

By Lemma~\ref{zero}, the expansion \eqref{def_krein} is simplified to
\begin{equation}
\varepsilon_{\lambda} \circ \varepsilon_{\mu}= 
\frac{\delta_{\eta(\lambda), \eta(\mu)}}{n_{\lambda}}
\sum_{\substack{\nu \in \Lambda \\ \eta(\nu)= \eta(\lambda) }}
q_{\lambda, \mu}^\nu \varepsilon_\nu. \label{krein_imp}
\end{equation}

For brevity, we write a basis of matrix units
$\{\varepsilon_{i,j}^s \mid i,j \in F_s \}$ as $\{\varepsilon_{i,j}^s \}$ and
we define $Z \circ \{\varepsilon_{i,j}^s \} 
= \{\zeta_{i,j} \varepsilon_{i,j}^s \mid i,j \in F_s  \}$
for a matrix $Z= (\zeta_{i,j}) \in \mathrm{M}_{F_s}(\mathbb{C})$.

\begin{lemma}\label{uniqueness}
Fix $s \in S$.
Let $Z= (\zeta_{i,j}) \in \mathrm{M}_{F_s}(\mathbb{C})$.
If $Z\circ \{\varepsilon_{i,j}^s \}$ is 
a basis of matrix units for $\mathfrak{C}_s$, then 
$Z$ is a positive semidefinite matrix with rank one and
$|\zeta_{i,j}|=1$ for all $i,j \in F_s$.
\end{lemma}
\begin{proof}
Since $ Z \circ \{\varepsilon_{i,j}^s \}$ is 
a basis of matrix units for $\mathfrak{C}_s$,
$Z \circ \{\varepsilon_{i,j}^s \}$ satisfies
\eqref{eq_mat_u}. This means that 
$\zeta_{i,j}^s \zeta_{j,k}^s= \zeta_{i,k}^s$ and
$\overline{\zeta_{j,i}^s} =\zeta_{i,j}^s$
for any $i,j,k \in F_s$. 
Thus $|\zeta_{i,j}|=1$ holds.
Moreover, 
Since $\zeta_{i,j}^s= \overline{\zeta_{1,i}^s} \zeta_{1,j}^s $ holds,
$Z$ is expressed as 
$
Z= \mathbf{z}^* \mathbf{z},
$
where $\mathbf{z}=( \zeta_{1,j})_{j \in F_s}. $
Thus $Z$ is a positive semidefinite matrix with rank one.
\end{proof}


\section{Fiber-commutative coherent configurations}

In this section, we also use the same notation as the previous section.
In other words, $\mathfrak{A}$ is the adjacency algebra of
a coherent configuration $\mathcal{C}$,
$\mathfrak{A}$ is decomposed into the direct sum of simple ideals as
$\mathfrak{A} = \bigoplus_{s \in S} \mathfrak{C}_s$.
Moreover $\{ \varepsilon_{i,j}^s\}$ is a basis of matrix units for $\mathfrak{C}_s$,
and their union over $s \in S$ is bases of matrix units for $\mathfrak{A}$.
In this section, we assume
that the coherent configuration $\mathcal{C}$ is fiber-commu\-tative.

\begin{lemma}\label{dim=1}
For any $s \in S$ and $k,l \in \{1 ,\dots , f\}$,
$\mathrm{dim} (\mathfrak{C}_s \cap \mathfrak{A}_{k,l}) \leq 1$.
In other words, the number of pairs $(i,j)$ satisfying
$\varepsilon_{i,j}^s \in \mathfrak{A}_{k,l}$ is at most $1$.
\end{lemma}
\begin{proof}
By \eqref{HWThm8}, for each $i,j \in F_s$,
there exist $k,l$ such that $\varepsilon_{i,j}^s \in \mathfrak{A}_{k,l}$.
Thus it suffices to show 
$\# \{(i,j,s) \in \Lambda_s \mid \varepsilon_{i,j}^s \in \mathfrak{A}_{k,l} \}
\leq 1$.
Suppose $\varepsilon_{i,j}^s , \varepsilon_{i'j'}^s \in \mathfrak{A}_{k,l}$
and $i \neq i'$.
By \eqref{diag_block}, we have 
$\varepsilon_{i,i}^s , \varepsilon_{i',i'}^s \in \mathfrak{A}_k$.
Thus $\varepsilon_{i,i'}^s =
\varepsilon_{i,i} \varepsilon_{i,i'} \varepsilon_{i',i'} \in \mathfrak{A}_k$
holds.
Since $\mathfrak{A}_k$ is commutative,
$\varepsilon_{i,i'}^s = \varepsilon_{i,i}^s \varepsilon_{i,i'}^s
=\varepsilon_{i,i'}^s \varepsilon_{i,i}^s = 0$, 
and this is a contradiction.
Therefore, we obtain $i = i'$ and similarly, $j =j'$.
\end{proof}

Note that Lemma~\ref{dim=1} is stated implicitly by Hobart and Williford
in the proof of \cite[Corollary~10]{HW}. 
Since $\eta|_{\Lambda_s}:F_s^2 \times \{s\} \rightarrow \{1, \dots ,f \}^2$ is
injective by Lemma~\ref{dim=1},
the set $F_s$ can be taken to be the subset $F'_s$ of $\{1, \dots ,f \}$
defined in Lemma~\ref{lem:F'}. 

\begin{definition}
For $s \in S$, we define the
{\it support} for $\mathfrak{C}_s$ to be the subset
\[
F_s =\{i \in \{ 1, \dots , f\} \mid
\mathrm{dim}(\mathfrak{C}_s \cap \mathfrak{A}_{i,i}) =1 \}.
\]
\end{definition}

By the definition of $F_s$, we can take $\eta$ as
$\eta(i,j,a)= (i,j)$ for $i,j \in F_s$.
Indeed, by \eqref{diag_block}, we may suppose
$\varepsilon_{i,i}^s \in \mathfrak{A}_{i,i}$ for all $i\in F_s$.
Then by $\varepsilon_{i,j}^s =\varepsilon_{i,i}^s
\varepsilon_{i,j}^s \varepsilon_{j,j}^s \in \mathfrak{A}_{i,j},$
we have $\eta(i,j,s)=(i,j)$.

For brevity, we write $F_{s,t,u} =F_s \cap F_t \cap F_u$. 
Note that $F_1 =\{ 1, \dots , f\}$ holds by \eqref{matunit}.
By Lemma~\ref{zero}, \eqref{krein_imp} can be written as follows:
for $(i,j) \in F_s^2 \cap F_t^2$,
\begin{equation}
\varepsilon_{i,j}^s \circ \varepsilon _{i,j}^t
= \frac{1}{\sqrt{|X_i||X_j|}}\sum_{ \substack{u \in S \\ F_u \ni i,j} }
q_{{(i,j,s)}, (i,j,t)}^{(i,j,u)}\varepsilon_{i,j}^u. \label{krein2}
\end{equation}

\begin{definition}\label{def_kre_mat}
For $s,t,u \in S$,
let $Q_{s,t}^u \in \mathrm{M}_{F_{s,t,u}} (\mathbb{C})$
be the matrix
with $(i,j)$-entry 
\[
(Q_{s,t}^u)_{i,j} = q_{(i,j,s), (i,j,t)}^{(i,j,u)}.
\] 
The matrix $Q_{s,t}^u$ is called {\it the matrix of Krein parameters}
with respect to the bases of matrix units
$\{\varepsilon_{i,j}^s \},  \{\varepsilon_{i,j}^t \},\{\varepsilon_{i,j}^s \}$
for $\mathfrak{C}_s, \mathfrak{C}_t, \mathfrak{C}_u$.
\end{definition}

Note that, by \eqref{krein2}, $Q_{s,t}^u = Q_{t,s}^u$ holds
for any $s,t,u \in S$.
Moreover, the matrix $Q_{s,t}^u$ is hermitian
by \eqref{eq_mat_u2} and \eqref{krein2}.

\begin{proposition}\label{prop:1st}
For any $s,t \in S$, we have
$Q_{1,s}^t = \delta_{s,t} J_{F_{1,s,t}}$.
\end{proposition}
\begin{proof}
Immediate from \eqref{matunit}, \eqref{krein2} and Definition~\ref{def_kre_mat}.
\end{proof}

\begin{proposition}\label{prop:st1}
For any $s,t \in S$,
\[
Q_{s,t}^1
=\delta_{\hat{s},t} \mathrm{tr}(\varepsilon_{j,j}^t) J_{F_{s,t,1}}.
\]
In particular, $\mathrm{tr}(\varepsilon_{j,j}^t)$ is independent of $j\in F_t$.
\end{proposition}
\begin{proof}
By \eqref{krein2},
\begin{align*}
\left(\varepsilon_{i,j}^s \circ \varepsilon _{i,j}^t
\right) \sum_{k,l \in F_1} \varepsilon_{k,l}^1
&= \frac{1}{\sqrt{|X_i||X_j|}}
\sum_{k,l \in F_1} \sum_{ \substack{u \in S \\ F_u \ni i,j} }
(Q_{s,t}^u)_{i,j} \varepsilon_{i,j}^u \varepsilon_{k,l}^1 \\
&=\frac{1}{\sqrt{|X_i||X_j|}} (Q_{s,t}^1)_{i,j} \sum_{l\in F_1}
\varepsilon_{i,l}^{1}.
\end{align*}
We compute the trace of each side of this identity.
By \eqref{matunit},
the trace of the right-hand side is $(Q_{s,t}^1)_{i,j}/{\sqrt{|X_i||X_j|}}$.
On the other hand, 
the trace of the left-hand side is
\begin{align*}
\mathrm{tr} \left((
\varepsilon_{i,j}^s \circ \varepsilon _{i,j}^t
) \sum_{k,l \in F_1}\varepsilon_{k,l}^1 \right)
&= \sum_{x,y \in X}
\left(\sum_{k,l \in F_1}\varepsilon_{i,j}^s \circ \varepsilon _{i,j}^t
\circ \varepsilon_{k,l}^1 \right)_{x,y} \\
&=\frac{1}{\sqrt{|X_i||X_j|}} \sum_{x,y \in X} (\varepsilon_{i,j}^s
\circ \varepsilon _{i,j}^t)_{x,y} \quad (\text{by \eqref{matunit}}) \\
&=\frac{1}{\sqrt{|X_i||X_j|}}
\mathrm{tr} ({\varepsilon_{i,j}^s}^T
\varepsilon _{i,j}^t) \\
&=\frac{1}{\sqrt{|X_i||X_j|}}
\mathrm{tr} (\overline{\varepsilon_{j,i}^s}
\varepsilon _{i,j}^t)  \quad 
(\text{by } {\varepsilon_{i,j}^s}^* = \varepsilon_{j,i}^s) \\
&=\frac{1}{\sqrt{|X_i||X_j|}}
\mathrm{tr} (\varepsilon_{j,i}^{\hat{s}}
\varepsilon _{i,j}^t) \\
&=\frac{1}{\sqrt{|X_i||X_j|}} \delta_{\hat{s},t} 
\mathrm{tr} (\varepsilon _{j,j}^t). \\
\end{align*}
By the properties of the trace,
$\mathrm{tr}({\varepsilon_{i,j}^s}^T\varepsilon_{i,j}^t)
= \mathrm{tr}(\varepsilon_{i,j}^t {\varepsilon_{i,j}^s}^T) $ and
this implies
$\mathrm{tr}(\varepsilon_{j,j}^t) = \mathrm{tr}(\varepsilon_{i,i}^t)$.
Thus we obtain $(Q_{s,t}^1)_{i,j}= \delta_{\hat{s},t} \mathrm{tr}(\varepsilon_{i,i}^t)
= \delta_{\hat{s},t} \mathrm{tr}(\varepsilon_{j,j}^t)$,
and the result follows.
\end{proof}

\begin{proposition}
For $s,t,u \in S, $ let $\mathbf{z}_s \in \mathbb{C}^{F_s},
\mathbf{z}_t \in \mathbb{C}^{F_t}, \mathbf{z}_u \in \mathbb{C}^{F_u}$ be
vectors whose entries consist of complex numbers with absolute value $1$.
Define $\mathbf{z} \in \mathbb{C}^{F_{s,t,u}}$ by
\[
(\mathbf{z})_k = \frac{(\mathbf{z}_s)_k (\mathbf{z}_t)_k}{(\mathbf{z}_u)_k} \quad 
(k \in F_{s,t,u}).
\]
Then $\mathbf{z}^* \mathbf{z} \circ Q_{s,t}^u$
is the matrix of Krein parameters with respect to
$\mathbf{z}_s^* \mathbf{z}_s \circ \{\varepsilon_{i,j}^s\},
\mathbf{z}_t^* \mathbf{z}_t \circ \{\varepsilon_{i,j}^t\},
\mathbf{z}_u^* \mathbf{z}_u \circ \{\varepsilon_{i,j}^u\}$.
\end{proposition}
\begin{proof}
By \eqref{krein2} and Definition~\ref{def_kre_mat}, we have
\begin{align*}
&(\mathbf{z}_s^* \mathbf{z}_s)_{i,j} \varepsilon_{i,j}^s
\circ (\mathbf{z}_t^* \mathbf{z}_t)_{i,j} \varepsilon_{i,j}^t
\\&= (\mathbf{z}_s^* \mathbf{z}_s)_{i,j}(\mathbf{z}_t^* \mathbf{z}_t)_{i,j}
(\varepsilon_{i,j}^s \circ \varepsilon_{i,j}^t) \\
&=\frac{ (\mathbf{z}_s^* \mathbf{z}_s)_{i,j}(\mathbf{z}_t^* \mathbf{z}_t)_{i,j}}{
\sqrt{|X_i||X_j|} }
 \sum_{ \substack{u \in S \\ F_u \ni i,j} }
(Q_{s,t}^u)_{i,j}\varepsilon_{i,j}^u \\
&= \frac{1}{\sqrt{|X_i||X_j|} }
\sum_{ \substack{u \in S \\ F_u \ni i,j} }
\frac{(\mathbf{z}_s^* \mathbf{z}_s)_{i,j}(\mathbf{z}_t^* \mathbf{z}_t)_{i,j}}
{(\mathbf{z}_u^* \mathbf{z}_u)_{i,j}}
(Q_{s,t}^u)_{i,j} (\mathbf{z}_u^* \mathbf{z}_u)_{i,j}\varepsilon_{i,j}^u.
\end{align*}
Thus the result follows.
\end{proof}

In particular, if $Q_{s,t}^u$ is positive semidefinite,
then $Z \circ Q_{s,t}^u$ is also positive semidefinite. 
Thus the positive semidefiniteness of $Q_{s,t}^u$ is independent of the choice of
bases of matrix units.

\begin{theorem}\label{thm1}
For any $s,t,u \in S$,
the condition \eqref{krein_con} holds 
if and only if
the matrix of Krein parameters $Q_{s,t}^u$ is positive semidefinite.
\end{theorem}

\begin{proof}
To prove this equivalence, we simplify \eqref{krein_con}.
Let $B = (b_{i,j}) \in \mathcal{P}_{F_s}, 
C =(c_{i,j}) \in \mathcal{P}_{F_t}$.
By Lemma~\ref{zero}, if $(m,n) \not \in F_s^2$ or
$(m,n) \not \in F_t^2$, then the $(m,n)$-entry \eqref{tilde}
of $\tilde{Q}_{s,t}^u(B,C)$ is $0$. 
If $(m,n) \in F_{s,t,u}^2$, then \eqref{tilde} is 
\[
 b_{m,n} c_{m,n}(Q_{s,t}^u)_{m,n} 
= (B' \circ C' \circ Q_{s,t}^u)_{m,n},
\]
where $B',C' \in \mathrm{M}_{F_{s,t,u}}(\mathbb{C})$ are
the principal submatrices of $B,C$ indexed by $F_{s,t,u}$.
Thus $\tilde{Q}_{s,t}^u(B,C)$ has $B' \circ C' \circ Q_{s,t}^u $
as a principal submatrix and all other entries are $0$.
This implies that $\tilde{Q}_{s,t}^u(B,C) \in \mathcal{P}_{F_u}$
if and only if
$B' \circ C' \circ Q \in \mathcal{P}_{F_{s,t,u}}$.
In particular, taking $B$ and $C$ to be the all-ones matrices,
\eqref{krein_con} implies $Q_{s,t}^u \in \mathcal{P}_{F_{s,t,u}}$.

Conversely, if $Q_{s,t}^u \in \mathcal{P}_{F_{s,t,u}}$, then
$B' \circ C' \circ Q_{s,t}^u \in \mathcal{P}_{F_{s,t,u}}$
for any $B \in \mathcal{P}_{F_s}, C \in \mathcal{P}_{F_t}$
by \cite[Lemma~3.9]{BI},
and \eqref{krein_con} holds.
\end{proof}

Hobart \cite{Ho} applied the Krein condition of the coherent configuration
defined by a quasi-symmetric design by setting $B$ and $C$ to be all-ones
matrices. She commented that there are no choices of $B,C$ 
which lead to other consequences.
Indeed, since the coherent configuration defined by a quasi-symmetric design
is fiber-commu\-tative,
considering the case $B=C=J$ is sufficient by Theorem~\ref{thm1}.

\section{Generalized quadrangles}

\begin{definition}
Let $P,L$ be finite sets and 
$I \subset P \times L$ be an incidence relation.
An incidence structure $(P,L,I)$ is 
called {\it a generalized quadrangle with parameters $(s,t)$} if 
\begin{enumerate}
\item for any $l \in  L$, $\# \{p \in P \mid (p,l )\in I \}=s+1$,
\item for any $ p \in P$, $\#\{\l \in L \mid (p,l)\in I \}= t+1$,
\item for any $p \in P$ and $l\in L$ with $(p,l) \not\in I$,
there exist unique $q \in P$ and unique $m \in L$ such that
$(p,m), (q,m) ,(q,l) \in I$.
\end{enumerate}
Elements of $P$ and $L$ are called {\it points } and {\it lines}, respectively.
\end{definition}

Let $(P,L,I)$ be a generalized quadrangle with parameters $(s,t)$. 
For $p,q \in P$, if there exists $l \in L$ such that $(p,l),(q,l) \in I$,
then we write $p \sim q$ and say that $p$ and $q$ are {\it collinear.}
Similarly, for $l,m \in L$, if there exists $p \in P$ such that
$(p,l),(p,m) \in I$,
then we write $l \sim m$ and say that $l$ and $m$ are {\it concurrent.}

In this section, we apply Theorem~\ref{thm1} to generalized quadrangles and
obtain the following inequalities:
If $s,t >1$, then $s \leq t^2$ and $t \leq s^2$ hold.
These inequalities are established in \cite{Hi3,Hi2}, as a consequence of
the Krein condition for the strongly regular graph defined by a generalized
quadrangle.
We also show that no other consequences can be obtained from Theorem~\ref{thm1}
by computing all matrices of Krein parameters.

First, we construct a coherent configuration from a generalized quadrangle.
Let $X_1=P$ and $X_2= L$ be fibers.
Adjacency relations on $X= X_1 \sqcup X_2$ are defined as
\begin{align*}
R_{1,1,1} &= \{(p,p) \mid p \in P \}, \\
R_{1,1,2} &= \{(p,q) \in P^2 \mid p \sim q,\; p \neq q  \}, \\
R_{1,1,3} &= \{(p,q) \in P^2 \mid p \not \sim q \}, \\
R_{1,2,1} &= \{(p,l) \in P \times L \mid (p,l) \in I \}, \\
R_{1,2,2} &= \{(p,l) \in P \times L \mid (p,l) \not \in I \}, \\
R_{2,1,1} &= \{(l,p) \in L \times P \mid (p,l) \in I \}, \\
R_{2,1,2} &= \{(l,p) \in L \times P \mid (p,l) \not \in I\}, \\
R_{2,2,1} &= \{(l,l) \mid l \in L \}, \\
R_{2,2,2} &= \{(l,m) \in L^2 \mid l \sim m,\; l \neq m \}, \\
R_{2,2,3} &= \{(l,m) \in L^2 \mid l \not\sim m \}.
\end{align*}
Then $\mathcal{C}=(X,\{R_I\}_{I \in \mathcal{I}})$ is a coherent configuration,
where $\mathcal{I} = \{(i,j,k) \mid 1 \leq i,j \leq 2,\; 1 \leq k \leq r_{i,j} \}$
and $r_{1,1}=r_{2,2}=3$, $r_{1,2}=r_{2,1}=2$.
Let $A_{i,j,k}$ be the adjacency matrix of the relation $R_{i,j,k}$, and let
$\mathfrak{A}$ be the adjacency algebra of $\mathcal{C}$.
Then $\mathfrak{A}$ is decomposed as
\[
\mathfrak{A} = \mathfrak{C}_1 \oplus \mathfrak{C}_2 \oplus \mathfrak{C}_3
\oplus \mathfrak{C}_4,
\]
where $\mathfrak{C}_1 , \mathfrak{C_2} \simeq \mathrm{M}_2 (\mathbb{C})$
and $\mathfrak{C}_3, \mathfrak{C}_4 \simeq \mathbb{C}$.
Moreover, $F_1 =F_2 = \{1,2\}$, $F_3=\{1 \}$, $F_4 =\{2 \}$.

For each $\mathfrak{C}_i$, a basis of matrix units can be expressed as follows:
For $\mathfrak{C}_1$,
\begin{align*}
\varepsilon_{1,1}^{1} &= \frac{1}{(st+1)(s+1)} (A_{1,1,1}+A_{1,1,2}+A_{1,1,3}), \\
\varepsilon_{2,2}^{1} &= \frac{1}{(st+1)(t+1)} (A_{2,2,1}+A_{2,2,2}+A_{2,2,3}), \\
\varepsilon_{1,2}^{1} &= \frac{1}{(st+1)\sqrt{(s+1)(t+1)}} (A_{1,2,1}+A_{1,2,2}), \\
\varepsilon_{2,1}^{1} &= \frac{1}{(st+1)\sqrt{(s+1)(t+1)}} (A_{2,1,1}+A_{2,1,2}). \\
\end{align*}
For $\mathfrak{C}_2$,
\begin{align*}
\varepsilon_{1,1}^{2} &= \frac{1}{(st+1)(s+t)} 
(st(t+1)A_{1,1,1}+t(s-1)A_{1,1,2}-(t+1) A_{1,1,3}), \\
\varepsilon_{2,2}^{2} &= \frac{1}{(st+1)(s+t)}
(st(s+1) A_{2,2,1}+s(t-1) A_{2,2,2}-(s+1)A_{2,2,3}), \\
\varepsilon_{1,2}^{2} &= \frac{1}{(st+1)\sqrt{(s+t)}}
(stA_{1,2,1}-A_{1,2,2}), \\
\varepsilon_{2,1}^{2} &= \frac{1}{(st+1)\sqrt{(s+t)}}
(stA_{2,1,1}-A_{2,1,2}). \\
\end{align*}
For $\mathfrak{C}_3$,
\[
\varepsilon_{1,1}^{3} = \frac{1}{(s+t)(s+1)} (s^2A_{1,1,1}-sA_{1,1,2} +A_{1,1,3}).
\]
For $\mathfrak{C}_4$,
\[
\varepsilon_{2,2}^{4} = \frac{1}{(s+t)(t+1)} (t^2A_{2,2,1}-tA_{2,2,2} +A_{2,2,3}).
\]

For these bases of matrix units,  the matrices of Krein parameters $Q_{3,3}^3$ and $Q_{4,4}^4$ are
the $1 \times 1$ matrices given by
\begin{align*}
Q_{3,3}^3 &= \frac{(st+1)(s-1)(s^2-t)}{(s+t)^2}, \\
Q_{4,4}^4 &= \frac{(st+1)(t-1)(t^2-s)}{(s+t)^2}.
\end{align*}
By Theorem~\ref{thm1},
both $Q_{3,3}^3$ and $Q_{4,4}^4$ are positive semidefinite,
so $s^2 \geq t$ and $t^2 \geq s$ hold, provided $s,t>1$.
The consequences of Theorem~\ref{thm1} for all other matrices
of Krein parameters are trivial. 
Indeed, the other matrices of Krein parameters are given as follows
(we omit those matrices determined by Proposition~\ref{prop:1st},
and those determined to be zero by Proposition~\ref{prop:st1}):
\begin{align*}
Q_{2,2}^1 &= \frac{st(s+1)(t+1)}{(s+t)} 
\begin{bmatrix}
1 & 1  \\
1& 1
\end{bmatrix},\\
Q_{2,2}^2 &= \frac{1}{(s+t)^2}
\begin{bmatrix}
\sigma(s,t) & \tau(s,t)  \\
\tau(s,t)& \sigma(t,s)
\end{bmatrix},
\end{align*}
where
\begin{align*}
\sigma(s,t)&=(s+1)(t^2 (st+2s-1) +s(st-2t-1)),\\
\tau(s,t)&=(s+t)^{3/2}(st-1)\sqrt{(s+1)(t+1)},
\end{align*}
and

\begin{tabular}{ll}
$\displaystyle Q_{2,2}^3 = \frac{t(st+1)(s+1)(t+1)}{(s+t)^2}$,
& $\displaystyle Q_{2,2}^4 = \frac{s(st+1)(s+1)(t+1)}{(s+t)^2}$, \\
$\displaystyle Q_{2,3}^2 = \frac{s(st+1)^2}{(s+t)^2}$, 
& $\displaystyle Q_{2,3}^3 = \frac{t(t+1)(s+1)^2(s-1)}{(s+t)^2}$, \\
$\displaystyle Q_{2,4}^2 = \frac{t(st+1)^2}{(s+t)^2}$, 
& $\displaystyle Q_{2,4}^4 = \frac{s(s+1)(t+1)^2(t-1)}{(s+t)^2}$, \\
$\displaystyle Q_{3,3}^1 =\frac{s^2(st+1)}{(s+t)}$,
&$\displaystyle Q_{3,3}^2 = \frac{s(st+1)(s+1)(s-1)}{(s+t)^2}$, \\
$\displaystyle Q_{4,4}^1 =\frac{t^2(st+1)}{(s+t)}$,
&$\displaystyle Q_{4,4}^2 = \frac{t(st+1)(t+1)(t-1)}{(s+t)^2}$.
\end{tabular}


\section{Absolute bounds for fiber-commutative coherent configurations}

Let $\mathfrak{A}$ be the adjacency algebra of
a coherent configuration $\mathcal{C}=(X,\{R_I\}_{I\in\mathcal{I}})$, and let
$\{\Delta_s \mid s\in S \}$ be a set of representatives
of all irreducible matrix representations of $\mathfrak{A}$ over $\mathbb{C}$
satisfying $\Delta_s(A)^*= \Delta_s(A^*)$ for any $A \in \mathfrak{A}$.
Denote by $h_s$ the multiplicity of $\Delta_s$ 
in the standard module $\mathbb{C}^X$.
In this section, we assume that $\Delta_s(\varepsilon_{i,j}^s)=E_{i,j}$
for a basis of matrix units
$\{ \varepsilon_{i,j}^s \}$ for $\mathfrak{C}_s$,
where $E_{i,j} $ is $e_s \times e_s$ matrix
with $(i,j)$-entry $1$ and all other entries $0$.
The following bound is known as the absolute bound.

\begin{lemma}[{\cite[Theorem~5]{HW}}] \label{abs}
For any $s,t \in S$, we have
\[
\sum_{u \in S} h_u \mathrm{rank}
\left(
\sum_{\lambda \in \Lambda_s} \sum_{\mu \in \Lambda_t} \sum_{\nu \in \Lambda_u}
q_{\lambda,\mu}^\nu \Delta_u(\varepsilon_\nu)
\right)
\leq 
\begin{cases}
h_s h_t & \text{ if } s \neq t, \\
\binom{h_s+1}{2} & \text{ if } s=t.
\end{cases}
\]
\end{lemma}

For fiber-commu\-tative coherent configurations, we can simplify this inequality.

\begin{theorem}
Let $Q_{s,t}^u$ $(s,t,u\in S)$ be the 
matrices of Krein parameters for $\mathcal{C}$.
For any $s,t \in S$, we have
\[
\sum_{u \in S} h_u \mathrm{rank}(Q_{s,t}^u)
\leq 
\begin{cases}
h_s h_t & \text{ if } s \neq t, \\
\binom{h_s+1}{2} & \text{ if } s=t.
\end{cases}
\]
\end{theorem}

\begin{proof}
By \eqref{krein_imp}, for any $u \in S$, we have
\begin{align*}
\sum_{\lambda \in \Lambda_s} \sum_{\mu \in \Lambda_t} \sum_{\nu \in \Lambda_u}
q_{\lambda,\mu}^\nu \Delta_u(\varepsilon_\nu)
&=\sum_{i,j \in F_u}
q_{(i,j,s),(i,j,t)}^{(i,j,u)} \Delta_u(\varepsilon_{i,j}^u) \\
&= \sum_{i,j \in F_u}
(Q_{s,t}^u)_{i,j} E_{i,j},
\end{align*}
and the rank of this matrix is $\mathrm{rank}(Q_{s,t}^u)$.
By Lemma~\ref{abs}, the result follows.
\end{proof}


\end{document}